\documentclass[a4 paper,12pt,bezier]{article}
\usepackage[top=3cm,bottom=3cm,left=2.5cm,right=2.5cm]{geometry}
\usepackage{amsmath}
\usepackage{amsfonts,amsthm,amssymb}
\usepackage{amsfonts}
\usepackage{graphics,subfigure,caption2}
\usepackage{lmodern,amsmath,amssymb,amsthm,booktabs,longtable,array}
\usepackage{authblk}%%%for authors
\usepackage{tikz, color}
\usepackage{graphics,epstopdf}
\usepackage{latexsym,bm}
\usepackage{amsfonts,amsthm,amssymb,mathrsfs,bbding}
\usepackage{CJK}
\usepackage{indentfirst}
\usepackage{enumerate}

\usepackage{color}

\newtheorem{lem}{Lemma}[section]
\newtheorem{thm}[lem]{Theorem}
\newtheorem{cor}[lem]{Corollary}
\newtheorem{con}[lem]{Conjecture}
\newtheorem{prob}[lem]
{Problem}

\newtheorem{cl}{Claim}
\theoremstyle{plain}

\usepackage{cite}

\newcounter{countclaim}

\def \proof{\noindent {\bf Proof}. }
\newcommand {\proofend}{{\hfill$\Box$}}

\def \gap {\vspace{5mm}}

\def \rmhc{\mathrm{HC}}

\def \clproof{\noindent {\it Proof}. }

\newcommand {\clproofend}
{{\hfill$\natural$}}

\newcommand{\cP}{\mathcal P}
\newcommand{\cQ}{\mathcal Q}

\def \setp {{\mathcal P}}
\def \setq {{\mathcal Q}}

\begin{document}
	
	\baselineskip 6.2mm
	
	\title{Every $3$-connected $\{K_{1,4},K_{1,4}+e\}$-free split graph of order at least $13$ is Hamilton-connected}

	\author[1,2,3]{\small Tao Tian\thanks{ Email: taotian0118@163.com}}
	
	\author[4]{\small Fengming Dong\thanks{Corresponding author. Email: fengming.dong@nie.edu.sg and donggraph@163.com}}

\affil[1]{\footnotesize School of Mathematics and Statistics, Fujian Normal University, Fuzhou 350117, China}
		
\affil[2]{\footnotesize Key Laboratory of Analytical Mathematics and Applications (Ministry of Education), Fujian Key Laboratory of Analytical
Mathematics and Applications (FJKLAMA), Fujian Normal University, Fuzhou 350117, China}
			
\affil[3]{\footnotesize
Center for Applied Mathematics of Fujian Province (FJNU), Fujian Normal University, Fuzhou 350117, China}
		
\affil[4]{\footnotesize National Institute of Education, Nanyang Technological University, Singapore}

\date{}

\maketitle

\begin{abstract}
A graph is split if its vertex set can be partitioned into a clique and an independent set. Let \(K_{1,4}+e\) denote the graph obtained from \(K_{1,4}\) by adding an edge between two leaves. Ryj\'{a}\v{c}ek et al. [J. Combin. Theory Ser. B 134 (2019), 239--263] conjectured that every $4$-connected graph with minimum degree at least 6 that contains neither \(K_{1,4}\) nor \(K_{1,4}+e\) as an induced subgraph is Hamiltonian. They proved the conclusion for $5$-connected graphs with the same minimum degree bound. We prove that every $3$-connected split graph of order at least 13 that contains neither of these induced subgraphs is Hamilton-connected.
Together with Dirac's theorem, this proves the conjecture of
Ryj\'{a}\v{c}ek et al. for split
graphs.
\end{abstract}


%\large \tableofcontents

\section{Introduction }

In this paper, we consider only
finite simple graphs
 and refer the reader to \cite{24} for undefined notation and terminology.
For any graph $G$,
let $V(G)$ and $E(G)$ denote
its vertex set and edge set,
respectively.
A graph $G$ is $\{F_{1}, F_{2},\dots,F_{k}\}$-free if
it has no induced subgraph isomorphic to $F_i$
for any $i\in[k]$,
where $[k]=\{1,2,\ldots,k\}$.
A path with endpoints
$u, v$ is called a $(u, v)$-path.
It is {\it Hamiltonian} if it spans $G$.
A graph $G$ is {\sl Hamilton-connected}
if for
every pair of distinct vertices $u,v\in V(G)$, there exists a
Hamiltonian  $(u,v)$-path in $G$.
If $|V (G)| \ge 3$, a Hamiltonian path between the endpoints of an edge, together with that edge, forms a
{\it Hamiltonian cycle}.
Every Hamilton-connected graph of order at least three is Hamiltonian. Indeed, a
Hamiltonian path between the endpoints of any edge forms a Hamiltonian cycle
together with that edge.

%Obviously, a Hamilton-connectedgraph of order at least three is Hamiltonian.

A graph is a split graph if its vertex set can be partitioned
into a clique and an independent set.
Split graphs were introduced by Foldes and Hammer \cite{10}  in 1977 and have been further studied in \cite{1,9,11,12,13,14,15,29,30,31}.

\begin{thm} {\rm (Renjith and Sadagopan  \cite{9})}\label{Th4}
Let $G$ be a  $K_{1,3}$-free split graph. Then $G$ is Hamiltonian if and only if $G$ is $2$-connected.
\end{thm}

Dai et al. \cite{8}
strengthened Theorem~\ref{Th4} by proving pancyclicity.
A graph $G$ of order $n\ge 3$
is {\sl pancyclic} if it contains a cycle of every length $\ell$
with $3\le \ell \le n$.

\begin{thm} {\rm (Dai et al.  \cite{8})}\label{Th3}
Let $G$ be a  $K_{1,3}$-free split graph with $|V(G)|\ge 3$. 
Then $G$ is pancyclic if and only if $G$ is $2$-connected.
\end{thm}

\begin{thm} {\rm (Dai et al.  \cite{8})}\label{Th2}
If $G$ is a $3$-connected $K_{1,4}$-free split graph, then $G$ is Hamiltonian.
\end{thm}

Recently, Liu et al. \cite{1} explored  the Hamilton-connectedness of $r$-connected $K_{1,r}$-free split graphs, where $r\in \{3,4\}$.

\begin{thm}{\rm (Liu et al.  \cite{1})} \label{th5}
For any $r\in \{3,4\}$, if
$G$ is an  $r$-connected $K_{1,r}$-free split graph,
then $G$  is Hamilton-connected.
\end{thm}

Liu et al. \cite[Conjecture 5.2]{1} posed the conjecture for
$r \ge 3$.
In view of Theorem~\ref{th5},
it remains to consider $r \ge 5$.

\begin{con} {\rm(Liu et al. \cite[Conjecture 5.2]{1})}
For every integer \(r\ge5\), every \(r\)-connected \(K_{1,r}\)-free split graph is Hamilton-connected.
\end{con}

\iffalse
\begin{thm}{\rm (Liu et al.  \cite{1})} \label {th6}
Let $G$ be a  $4$-connected $K_{1,4}$-free split graph. Then $G$  is Hamilton-connected.
\end{thm}
\fi

\begin{figure}%\label {Fig1}%[htbp]
  \centering
   %Requires \usepackage{graphicx}
 % \includegraphics[width=0.65\linewidth]{Figure1}
  
  \begin{tikzpicture}[
  	line width=.5pt,
  	vertex/.style={circle, fill=black, inner sep=0pt, minimum size=4.4pt},
  	graph label/.style={anchor=base west, inner sep=0pt, font=\Large},
  	panel label/.style={anchor=base, inner sep=0pt}
  	]
  	\begin{scope}
  		\node[vertex] (hub) at (0,0) {};
  		\foreach \x in {-.68,0,.72} {
  			\node[vertex] (leaf) at (\x,1.70) {};
  			\draw (hub) -- (leaf);
  		}
  		\node[graph label] at (.54,.36) {$K_{1,3}$};
  	%	\node[panel label] at (0,-.62) {(a)};
  	\end{scope}
  	
  	\begin{scope}[xshift=3.35cm]
  		\node[vertex] (hub) at (0,0) {};
  		\foreach \x in {-1.02,-.34,.32,1.02} {
  			\node[vertex] (leaf) at (\x,1.70) {};
  			\draw (hub) -- (leaf);
  		}
  		\node[graph label] at (.68,.36) {$K_{1,4}$};
  	%	\node[panel label] at (0,-.62) {(b)};
  	\end{scope}
  	
  	\begin{scope}[xshift=7.10cm]
  		\node[vertex] (hub) at (0,0) {};
  		\foreach \x [count=\i] in {-1.02,-.34,.32,1.02} {
  			\node[vertex] (leaf\i) at (\x,1.70) {};
  			\draw (hub) -- (leaf\i);
  		}
  		\draw (leaf3) -- (leaf4);
  		\node[graph label] at (.66,.36) {$K_{1,4}+e$};
  %		\node[panel label] at (0,-.62) {(c)};
  	\end{scope}
  \end{tikzpicture}

  (a) \hspace{2.55 cm} (b)
  \hspace{3.0 cm} (c) \hspace{1.5 cm}{}

  \caption{The graphs $K_{1,3},K_{1,4}~\text{and}~ K_{1,4}+e$.}
  \label{fig1}
\end{figure}

Let $r\geq 3$ be an integer and let $K_{1,r+1}+e$ be the graph obtained from $K_{1,r+1}$ by
joining two leaves of
$K_{1,r+1}$ by an edge.
The case $r = 3$ is shown in Fig. \ref{fig1}(c).
The purpose of this paper is to
study the
 Hamilton-connectedness of $\{K_{1,4},K_{1,4}+e\}$-free
 split graphs.
Since both $K_{1,4}$
and $K_{1,4}+e$ contain an induced $K_{1,3}$,
 a $K_{1,3}$-free graph is also a $\{K_{1,4},K_{1,4}+e\}$-free graph, but the reverse is not true. Let $\delta(G)$ denote the minimum degree of a graph $G$. The Hamiltonicity of $\{K_{1,4},K_{1,4}+e\}$-free graphs was first studied by Li and Schelp \cite{3}  in 2002.

\begin{thm} {\rm (Li  and Schelp \cite{3})} \label {Th7}
Let $G$ be a $3$-connected $\{K_{1,4},K_{1,4}+e\}$-free   graph of order $n\geq 30$. If $\delta(G)\geq (n+5)/5$, then $G$ is Hamiltonian.

\end{thm}

Later, Li \cite{2} considered the Hamiltonicity of $2$-connected $\{K_{1,4},K_{1,4}+e\}$-free   graphs.

\begin{thm} {\rm (Li \cite{2})} \label {Th8}
Let $G$ be a $2$-connected $\{K_{1,4},K_{1,4}+e\}$-free   graph of order $n\geq 13$. If $\delta(G)\geq n/4$, then $G$ is Hamiltonian or $G\in \mathcal F$, where $\mathcal F$ is a family of non-Hamiltonian graphs of connectivity $2$
described in \cite{2}.
\end{thm}

Ryj\'{a}\v{c}ek et al. \cite{6}
established a constant minimum-degree bound that guarantees Hamiltonicity in \(5\)-connected \(\{K_{1,4},K_{1,4}+e\}\)-free graphs.

\begin{thm} {\rm (Ryj\'{a}\v{c}ek et al. \cite{6})} \label {Th9}
Every  $5$-connected $\{K_{1,4},K_{1,4}+e\}$-free   graph with minimum degree at least 6 is Hamiltonian.

\end{thm}

Ryj\'{a}\v{c}ek et al. \cite{6} posed Conjecture~\ref{con2} and proved that it is
equivalent to Conjecture~\ref{con3} below.

\begin{con} {\rm (Ryj\'{a}\v{c}ek et al. \cite{6})} \label {con2}
Every  $4$-connected $\{K_{1,4},K_{1,4}+e\}$-free   graph with minimum degree at least 6 is Hamiltonian.

\end{con}

\begin{con} {\rm (Ryj\'{a}\v{c}ek et al. \cite{6})} \label {con3}
Every  $4$-connected line  graph with minimum degree at least 5 is Hamiltonian.

\end{con}

Obviously, Conjecture \ref{con3} is a weaker version of ~Thomassen's conjecture \cite{25}, which states that  every  $4$-connected line  graph is Hamiltonian. Ryj\'{a}\v{c}ek  \cite{27} showed that Thomassen's conjecture is equivalent to the Matthews-Sumner conjecture \cite{26}, which asserts that every  $4$-connected $K_{1,3}$-free graph is Hamiltonian. For more information on  equivalent conjectures, we refer the reader to \cite{28}.
For the traceability of $\{K_{1,4},K_{1,4}+e\}$-free   graphs, we refer the reader to \cite{4,5,7}.

There exist
infinitely many $3$-connected $K_{1,4}$-free split graphs that are not Hamilton-connected \cite{1}.
In this article, we focus on
$\{K_{1,4},K_{1,4}+e\}$-free
split graphs and establish
the following result on their
Hamilton-connectedness.

\begin{thm}\label {Th1}
If $G$ is a $3$-connected $\{K_{1,4},K_{1,4}+e\}$-free split graph of order at least $13$, then $G$ is Hamilton-connected.
\end{thm}

\begin{cor}\label{Th1-cor1}
	Conjecture~\ref{con2} holds for split graphs.
\end{cor}

\proof Let
$G$ be a $4$-connected $\{K_{1,4},K_{1,4}+e\}$-free split graph with $\delta(G)\ge 6$.
Theorem~\ref{Th1} shows that
$G$ is Hamilton-connected
if $|V(G)|\ge 13$.
If  \(|V(G)\le 12\), then \(\delta(G)\ge 6\ge |V(G)|/2\), so \(G\) is Hamiltonian by Dirac's theorem. Thus the corollary holds.
%Thus Conjecture ~\ref{con2} holds for split graphs.
\proofend

\iffalse
For graphs of order at least $13$, Theorem~\ref{Th1} extends the $r=3$ case of Theorem~\ref{th5} from claw-free split graphs to $\{K_{1,4},K_{1,4}+e\}$-free split graphs. The additional exclusion of $K_{1,4}+e$ is motivated by the existence of infinitely many $3$-connected $K_{1,4}$-free split graphs that are not Hamilton-connected \cite{1}.
\fi

 The remainder of this manuscript is organized as follows.
In Section 2, we recall several results of Liu et al. \cite{1}.
In Section 3, we establish two forbidden-subgraph lemmas for \(\{K_{1,4},K_{1,4}+e\}\)-free split graphs.
In Section 4, we prove the existence of an $I$-cover in a
 $3$-connected
$\{K_{1,4}, K_{1,4}+e\}$-free
 split graph $G$
with $|V (G)| \ge 9$.
 The proof of Theorem~\ref{Th1} is given in Section~5.
 Finally, in Section 6, we pose an open problem on the order threshold for Hamilton-connectedness and construct a \(22\)-vertex, \(5\)-connected \(\{K_{1,6},K_{1,6}+e\}\)-free split graph that fails to be Hamilton-connected.

\section{Preliminaries}

In this section, we  present some useful lemmas.

For two graphs $G_{1}$ and $G_{2}$, let $G_{1}\cup G_{2}$ be a graph with vertex set $V(G_{1})\cup V(G_{2})$ and edge set $E(G_{1})\cup E(G_{2})$.
For any $X\subseteq V(G)$,
let $G[X]$ denote
	the subgraph
of $G$ induced by $X$,
where $X\ne \emptyset$,
and let  $G-X=G[V(G)\setminus X]$, where $X\ne V(G)$.
 Similarly,
 for any $Y\subseteq E(G)$,
 let $G[Y]$ denote the
 subgraph of $G$ with edge set $Y$
 and vertex set consisting of the endpoints of the edges in $Y$,
 and let
 $G-Y=G[E(G)\setminus Y]$.
 For a subgraph $H$ of $G$ with $ E(H)\cap Y =\emptyset$,
 we write $H+Y$ for $H\cup G[Y]$.

Let $G$ be a split graph with $V(G)=S\cup I$, where $S$
 is a maximum clique of $G$ and $I$ is an independent set of $G$. Following Liu et al.  \cite{1}, we call such an ordered pair $(S,I)$ a {\sl split partition} of $G$. Denote $s=|S|$.
 Then, $G[S]\cong K_{s}$ and  $G[I]\cong |I| K_{1}$.
 Since any complete graph of order at least three is Hamilton-connected, we may assume that $|I|>0$. Let $I=\{v_{1},v_{2},\ldots,v_{|I|}\}$. Then $N_{G}(v)\subseteq S$ for any $v\in I$. If $\kappa(G)\geq k$, then $s\geq k$.   By the maximality of $s$, $k \leq d_{G}(v)\leq s-1$ for any $v\in I$. So, $s\geq k+1$.

 A path $P=a_{1}a_{2}\ldots a_{2t+1}$ in $G$,
 where $t\ge 1$ and $a_{2t+1}\ne a_1$,
 is called
 an {\sl $(S,I)$-alternating path}
 if  $V(P)\cap S=\{a_{1},a_{3},\ldots, a_{2t+1}\}$.
  Let $End(P)=\{a_1,a_{2t+1}\}$ and
  $P^{o}=V(P)\setminus End(P)$ be the {\sl interior} of the path $P$.
 Similarly,
 a cycle $C=a_{1}a_{2}\ldots a_{2t+1}$ in $G$,
 where $t\ge 2$ and  $a_{1}= a_{2t+1}$,
 is called
 an {\sl $(S,I)$-alternating cycle}
 if its odd-indexed
 vertices lie in $S$ and its even-indexed vertices lie in $I$.

We call a collection $\mathcal P=\{P_{1},P_{2},\ldots,P_{h}\}$ of $(S,I)$-alternating paths in $G$ an {\sl $I$-cover} if the following conditions hold:
\begin{enumerate}[(A)]
	\item $V(P_{i})\cap V(P_{j})=\emptyset$ for any $i,j\in [h]$ and $i\neq j$; and

\item $I\subseteq \bigcup_{i=1}^{h}P_{i}^{o}$.
\end{enumerate}

For a collection $\setp=\{P_{1},P_{2},\ldots,P_{h}\}$ of $(S,I)$-alternating paths,
denote by $End(\mathcal P)$ the union of $End(P_i)$'s
over all paths $P_i$ in $\setp$.
Let $Inn(\setp)=(\bigcup_{i=1}^{h}P_{i}^{o})\cap S$. For convenience, we also use $\mathcal P$ to denote the subgraph $\bigcup_{i=1}^{h}P_{i}$.

\begin{lem} {\rm (Liu et al.  \cite{1})} \label {lem9}
Let $G$ be a split graph with a split partition $(S, I)$ and an $I$-cover $\mathcal P$, and let $u, v$ be any two distinct vertices in
$V(G)\setminus Inn(\mathcal P)$.
If $End(P_i)\ne \{u, v\}$ for every
$P_i \in \mathcal P$, then $G$ has a Hamiltonian $(u, v)$-path.
\end{lem}

Let $G$ be a split graph with a split partition $(S,I)$.
Let $h_1, h_2\ge 0$ be integers,
and
let $P_{1},P_{2},\ldots,P_{h_{1}},
C_{1},C_{2},\ldots,C_{h_{2}}$
be pairwise vertex-disjoint  subgraphs of $G$.
We call the collection
$\mathcal Q=\{P_{1},P_{2},\ldots,P_{h_{1}},C_{1},C_{2},\ldots,C_{h_{2}}\}$
a {\sl pseudo $I$-cover} if the following conditions hold:
\begin{enumerate}
\item[(C)] each $P_{i}$ is an $(S,I)$-alternating path and each $C_{j}$ is an $(S,I)$-alternating cycle; and

\item[(D)] $I\subseteq (\bigcup_{i=1}^{h_{1}}P_{i}^{o})\cup (\bigcup_{j=1}^{h_{2}}V(C_{j}))$.
\end{enumerate}

For any pseudo $I$-cover $\setq$,  let
$End(\mathcal Q)$ denote
the set of all endpoints
 of all alternating paths in $\setq$,
and let $Inn(\mathcal Q)=(\bigcup_{i=1}^{h_{1}}P_{i}^{o})\cap S$.
For convenience, we also use $\mathcal Q$ to denote the subgraph $(\bigcup_{i=1}^{h_{1}}P_{i})\cup (\bigcup_{j=1}^{h_{2}}C_{j})$.

\begin{lem} {\rm (Liu et al.  \cite{1})} \label {lem3}
Let $r\geq 2$ be an integer, and let $G$ be an $r$-connected $K_{1,r+1}$-free split graph with a split partition $(S,I)$. Then $G$ contains a pseudo $I$-cover.
\end{lem}

\section{Forbidden subgraphs
}

Let $G$ be a split graph
with a split partition $(S,I)$.
The next two lemmas show that
if $G$ is   $\{K_{1,4},K_{1,4}+e\}$-free,
then $G$ does not contain
	any subgraph $H$
	with $|V(H)\cap S|=i+1$
	and $H\cong G_i$,
	where $1\le i\le 2$
	and
	$G_1$ and $G_2$ are the graphs shown in Fig.~\ref{fig2}.
These two conclusions
will be used in later proofs.

\begin{figure}[h!]
	\centering
	
\begin{tikzpicture}[
	line width=.45pt,
	vertex/.style={circle, fill=black, inner sep=0pt, minimum size=3.9pt},
	vertex label/.style={anchor=base, inner sep=0pt, font=\small}
	]
	% Left group: the stars with centres u and v.
	\node[vertex] (hub) at (0,0) {};
	\node[vertex label] at (0,-.35) {$u$};
	\foreach \x/\leafname in {-.64/a,0/b,.62/c} {
		\node[vertex] (leaf) at (\x,1.31) {};
		\draw (hub) -- (leaf);
		\node[vertex label] at (\x,1.48) {$\leafname$};
	}
	
	% The remaining four stars, in their order from left to right.
	\foreach \shift/\hubname/\leftname/\rightname in {
		2.12/v/x/y, 5.55/u/a/b, 7.67/v/c/d, 9.79/w/x/y
	} {
		\begin{scope}[xshift=\shift cm]
			\node[vertex] (hub) at (0,0) {};
			\node[vertex label] at (0,-.35) {$\hubname$};
			\foreach \x/\leafname in {-.64/\leftname,.62/\rightname} {
				\node[vertex] (leaf) at (\x,1.31) {};
				\draw (hub) -- (leaf);
				\node[vertex label] at (\x,1.48) {$\leafname$};
			}
		\end{scope}
	}
\end{tikzpicture}

	(a) $G_1$ \hspace{5 cm}
	(b) $G_2$
	
	\caption{
$V(G_1)\cap S=\{u,v\}$
		and $V(G_2)\cap S=\{u,v,w\}$.
	}
	\label{fig2}
\end{figure}

For any two graphs $H_1$ and $H_2$, let $H_1\cup_0 H_2$
denote the vertex-disjoint union of  $H_1$ and $H_2$.
For any positive integer $k$,
let $k H_1$
denote the vertex-disjoint union of $k$ copies of $H_1$.

\begin{lem}  \label {lem5}
	Let $G$ be a  $\{K_{1,4},K_{1,4}+e\}$-free split graph with a split partition $(S,I)$.
	Then, $G$ does not have
	a subgraph $H$
	such that $|V(H)\cap S|=2$
	and $H\cong K_{1,3}\cup_0 K_{1,2}$.
\end{lem}

\proof Suppose that $G$ contains
a subgraph $H$
such that $|V(H)\cap S|=2$
and $H\cong K_{1,3}\cup_0 K_{1,2}$.
We may assume that
$V(H)=\{u,v, a,b,c,x,y\}$,
with $V(H)\cap S=\{u,v\}$,
and
$E(H)=\{ua,ub,uc, vx,vy\}$,
as shown in Fig.~\ref{fig2}(a).

Since $u,v\in S$, we have $uv\in E(G)$.
Since $G$ is $K_{1,4}$-free and $a,b,c\in I$, we have $N_G(u)\cap I=\{a,b,c\}$. Since $v$ is adjacent to $x$ and $y$, it is adjacent to at most one of $a,b,c$. If it is adjacent to none of them, then $G[\{u,v,a,b,c\}]\cong K_{1,4}$. Otherwise, relabel so that $N_G(v)\cap\{a,b,c\}=\{a\}$. Then $G[\{v,u,a,x,y\}]\cong K_{1,4}+e$. Both alternatives are impossible.

Hence the result holds.
\proofend

\begin{lem}  \label {lem6}
	Let $G$ be a  $\{K_{1,4},K_{1,4}+e\}$-free split graph with a split partition $(S,I)$.
	Then, $G$ does not have
	a subgraph $H$
	such that $|V(H)\cap S|=3$
	and $H\cong 3 K_{1,2}$.
\end{lem}

\proof Suppose that $G$ contains
a subgraph $H$
such that $|V(H)\cap S|=3$
and $H\cong  3K_{1,2}$.
Each copy of \(K_{1,2}\) in \(H\) must contain a vertex of \(S\). Since \(|V(H)\cap S|=3\), each copy contains exactly one such vertex, which must be its center.
We may assume that
$V(H)=\{u,v, w, a,b,c,d, x,y\}$,
with $V(H)\cap S=\{u,v,w\}$,
and
$E(H)=\{ua,ub,vc, vd, wx,wy\}$,
as shown in Fig.~\ref{fig2}(b).

By Lemma~\ref{lem5},
we have $N_G(u)\cap \{c,d,x,y\}=\emptyset$.
Similarly, $N_G(v)\cap \{a,b,x,y\}=\emptyset$
and
$N_G(w)\cap \{a,b,c,d\}=\emptyset$.
Thus,
$
G[\{u,a,b,v,w\}]\cong K_{1,4}+e.
$

Since $u,v,w\in S$, we have
$G[\{u,v,w\}]\cong K_3$.
Thus, $G[V(H)]$ has an induced
subgraph isomorphic to $K_{1,4}+e$,
a contradiction to the assumption.

Hence the result holds.
\proofend

\section{Existence of an $I$-cover}

Let $G$ be a $3$-connected $\{K_{1,4},K_{1,4}+e\}$-free split graph with a split partition $(S,I)$. By Lemma \ref{lem3}, $G$ has a pseudo $I$-cover $\cQ=\{P_1,\ldots,P_{h_1},C_1,\ldots,C_{h_2}\}$. By Lemma \ref{lem6}, $h_2\le2$. In this section, we show that if $|V(G)|\ge9$, then $G$ has a pseudo $I$-cover with no alternating cycles, that is, an $I$-cover.

\begin{lem}  \label {lem7}
Let $G$ be a $3$-connected $\{K_{1,4},K_{1,4}+e\}$-free split graph with $|V(G)|\ge 9$
and a split partition $(S,I)$.
Then $G$ contains an  $I$-cover.
\end{lem}

\proof
Suppose, for a contradiction,
  that $G$ has no $I$-cover.
Let
  $\setq=\setq_1\cup \setq_2$
be a pseudo $I$-cover,
  where
$\setq_1
=\{P_{1},\ldots,P_{h_{1}}\}$
and $\setq_2
=\{C_{1},\ldots,C_{h_{2}}\}$,
such that
\begin{enumerate}[(1)]
\item $h_{2}$ is minimized; and
\item subject to (1), $h_{1}$ is minimized.
\end{enumerate}

By the assumption on $G$, $h_2\ge 1$. Then, by Lemma~\ref{lem6},
we have $h_{2}\leq 2$;
otherwise, $G$ contains a
subgraph $H$ with the properties that
$|V(H)\cap S|=3$ and $H\cong 3K_{1,2}$,
contradicting  Lemma~\ref{lem6}.

Let $S_{1}=S \cap V(\mathcal Q)$ and $S_{2}=S\setminus S_{1}$.

We first prove the following claim.

\setcounter{countclaim}{0}

\begin{cl} \rm \label {cl1}
For any $1\le i\le h_2$ and
$x\in I\cap V(C_{i})$, $N_{G}(x)\cap (S_{2}\cup End(\mathcal Q))=\emptyset$. So, $N_{G}(x)\subseteq S_{1}\setminus End(\mathcal Q).$
\end{cl}

\begin{figure}[h!]
	\centering

	\begin{tikzpicture}[
		x=1cm,
		y=1cm,
		every node/.style={inner sep=0pt, outer sep=0pt},
		vertex/.style={circle, fill=black, minimum size=3.2pt},
		ordinary edge/.style={draw=black, solid, line width=0.46pt, line cap=round},
		added edge/.style={ordinary edge, line width=1.15pt}
		]
		% The two copies of C_i and their added edges xy.
		\foreach \panelshift in {0,4.9} {
			\begin{scope}[shift={(\panelshift,0)}]
				\draw[ordinary edge]
				(0,1.5) -- (0,0) -- (1,1.5) -- (1,0) --
				(2,1.5) -- (2,0) -- (3,1.5) -- (3,0) -- cycle;
				\draw[added edge] (3,1.5) -- (4,0);
				
				\foreach \column in {0,1,2,3} {
					\node[vertex] at (\column,1.5) {};
					\node[vertex] at (\column,0) {};
				}
				\node[vertex] at (4,0) {};
				\node at (1.36,1.31) {$\cdots$};
				\node[above=3.5pt] at (3,1.5) {$x$};
				\node[below=3.5pt] at (4,0) {$y$};
				\node[below=3.5pt] at (1.36,0) {$C_i$};
			\end{scope}
		}
		
		% Panel (b) also contains P_j, whose first vertex is y.
		\begin{scope}[shift={(4.9,0)}]
			\draw[ordinary edge]
			(4,0) -- (5,1.5) -- (5,0) --
			(6,1.5) -- (6,0) -- (7,1.5) -- (7,0);
			\foreach \column in {5,6,7} {
				\node[vertex] at (\column,1.5) {};
				\node[vertex] at (\column,0) {};
			}
			\node at (5.31,1.16) {$\cdots$};
			\node[below=3.5pt] at (5.31,0) {$P_j$};
		\end{scope}
		
		\node at (1.60,-0.75) {(a)};
		\node at (8.65,-0.75) {(b)};
	\end{tikzpicture}

	\caption{The edge exchange in the proof of Claim 1. In each panel, one edge of \(C_i\) incident with \(x\) is removed from the pseudo \(I\)-cover, and the thick edge \(xy\) is added.
	}
	\label{lem4-1-f1}
\end{figure}

\clproof If
$y\in N_G(x)\cap S_2$,
then  $C_i$ can be replaced by
a new alternating path $P'$
with the property that  $V(P')=V(C_i)\cup \{y\}$,
as shown in Fig.~\ref{lem4-1-f1}(a),
and  a new pseudo $I$-cover
$\setq'$ can be obtained by
replacing $C_i$ by $P'$,
a contradiction to
the minimality of $h_2$.

Similarly, if
$y\in N_G(x)\cap End(P_j)$,
where $1\le j\le h_1$,
then  $C_i$ and $P_j$
can be replaced by
a new alternating path $P'$
with the property that  $V(P')=V(C_i)\cup V(P_j)$,
as shown in Fig.~\ref{lem4-1-f1}(b),
and  a new pseudo $I$-cover
$\setq'$ can be obtained by
replacing $C_i$ and $P_j$
by $P'$,
a contradiction to
the minimality of $h_2$.

Hence Claim~\ref{cl1} holds.
\clproofend

Using the conclusion of
Claim~\ref{cl1},
we now prove that $S_{2}=\emptyset$.

\begin{cl}
	\label {cl2}
  $S_{2}=\emptyset$.
\end{cl}

\clproof Suppose that $d\in S_{2}$. Let $C_{1}=a_{1}b_{1}a_{2}b_{2}\ldots a_{k}b_{k}a_{1}$, where $k\ge 2$,
  $a_i\in S$ and $b_i\in I$
  for each $i\in [k]$. Consider vertex $b_{1}$.  Since $\delta(G)\geq 3$, there is  a vertex $c\in S_{1}\backslash \{a_{1},a_{2}\}$ such that $b_{1}c\in E(G)$. By Claim \ref{cl1}, $c\notin End(\mathcal Q)$ and $b_{1}d\notin E(G)$.

 Write
  \[
  N_{\cQ}(c)=\{c_1,c_2\}.
  \]
  Here $c_1,c_2\in I$ are the two neighbors of $c$ in its component of $\cQ$. They are distinct from $b_1$, since $c\notin\{a_1,a_2\}$. Since $b_1d\notin E(G)$, the exclusion of both forbidden induced subgraphs implies $c_1d,c_2d\in E(G)$. If $c$ were on an alternating cycle, Claim 1 would instead give $c_1d,c_2d\notin E(G)$.
  Consequently, \(c\) is an interior vertex of some alternating path \(P_i\), where \(i\in[h_1]\).

Then
the graph $P_{i}\cup C_{1}-\{a_{1}b_{1},cc_{1}\}+\{b_{1}c,c_{1}d\}$
consists of two alternating paths,
a contradiction to the minimality of $h_{2}$.
Hence Claim~\ref{cl2} holds.
\clproofend

\begin{cl} \rm \label {cl3-0}
If $\setq_1\ne \emptyset$, then
for each $x\in I\cap V(\setq_2)$,
	$|N_G(x)\cap (S\cap V(\setq_2))|=2$,
	implying that
	$N_G(x)\cap (P^o_j\cap S)\ne \emptyset$ for some $P_j\in \setq_1$.
\end{cl}

\clproof
Assume that $\setq_1\ne \emptyset$.
Suppose there exists
$x\in I\cap V(\setq_2)$,
such that 	$|N_G(x)\cap (S\cap V(\setq_2))|\ne 2$.
Then,
$x$ is adjacent to some
vertex $y\in S$ which is on some cycle $C_r$ in $\setq_2$,
but $xy$ is not an edge in
$C_r$.

Let $u\in End(P_1)$.
As $y,u\in S$,
we have $yu\in E(G)$.
Let $x',x''$ be the two neighbors of $y$ on cycle $C_r$.
By Claim~\ref{cl1},
$\{x,x',x''\}\cap N_G(u)=\emptyset$.
Then, the subgraph induced by
$\{y, x, x',x'',u\}$ is isomorphic to $K_{1,4}$, contradicting the assumption that $G$ is $\{K_{1,4},K_{1,4}+e\}$-free.

Thus, $|N_G(x)\cap (S\cap V(\setq_2))|=2$.
Since $\delta(G)\ge 3$,
we have 	
$N_G(x)\cap (V(\setq_1)\cap S)\ne \emptyset$.
By Claim~\ref{cl1},
	$N_G(x)\cap (P^o_j\cap S)\ne \emptyset$ for some $P_j\in \setq_1$.
	
Hence Claim~\ref{cl3-0} holds.
\clproofend

It is known that $1\le h_2\le 2$.
%Now we are going to prove the following claim.

\begin{cl}\label{cl3-1}
	$h_2=1$.
\end{cl}

\clproof Suppose that $h_2=2$. By Lemma \ref{lem6}, every path $P_i\in\cQ_1$ has length 2; otherwise, its interior contains a clique vertex with two path neighbors in $I$, which together with the two cycles yields a forbidden $3K_{1,2}$. Hence $P_i^o\cap S=\varnothing$ for every $P_i\in\cQ_1$. If $\cQ_1\ne\varnothing$,
Claim~\ref{cl3-0} gives a neighbor in $P_j^o\cap S$ for a vertex of either cycle, a contradiction. Thus $\cQ_1=\varnothing$.

For $j=1,2$,
let $C_j$ be the alternating cycle $a_{j,1}b_{j,1}
\ldots a_{j,k_j}b_{j,k_j}a_{j,1}$,
where $a_{j,1},a_{j,2},$
$\ldots,a_{j,k_j}\in S$.
Without loss of generality, we assume that $k_{1}\geq k_{2}\geq2$.
Then, by Lemma \ref{lem6}, $ k_{1}\leq 3$.

Suppose that $k_{1}=3$.
Since $d_G(b_{1,1})\ge 3$,
$b_{1,1}$ is adjacent to some
vertex $a_{i,j}$ in $V(\setq_2)\cap S$, where
$i\in [2]$, but
$b_{1,1}a_{i,j}$ is not an edge of the pseudo $I$-cover.
%$b_{1,1}a_{i,j}$ is not an edge on cycle $C_i$.
Then, $G$ has
a subgraph $H$ with the property that $|V(H)\cap S|=2$
and $H\cong K_{1,3}\cup_0 K_{1,2}$, contradicting Lemma~\ref{lem5}.
Hence $k_1=2$.

By Claim~\ref{cl2}, $S_2=\emptyset$.
Since $\setq_1=\emptyset$,
we have $V(G)=V(\setq_2)$,
implying that $|V(G)|=|V(C_1)|+|V(C_2)|=8$,
contradicting the assumption that $|V(G)|\ge 9$.

  Hence Claim~\ref{cl3-1} holds.
  \clproofend

  By Claim~\ref{cl3-1},
  we may assume that
  $\setq_2=\{C\}$,
  where
  $C=a_{1}b_{1}a_{2}b_{2}\ldots a_{k}b_{k}a_{1}$, $k\ge 2$,
  $a_i\in S$ and $b_i\in I$
  for each $i\in [k]$.

  We now prove that $\setq_1=\emptyset$.

\begin{cl}\label{cl3-2}
$\setq_1= \emptyset$.
\end{cl}

\clproof
Suppose that $\setq_1\ne  \emptyset$.
We first show that $k<3$.
Otherwise, $k\ge 3$.
By Claim~\ref{cl3-0},
$N_G(b_1)\cap \{a_i: i\in [k]\}
=\{a_1,a_2\}$
and $b_1y\in E(G)$ for some
vertex $y$ in $\cup_{P_j\in \setq_1}P_j^o$.
It follows that $G$ has a subgraph $H$ with
$V(H)\cap S=\{y,a_3\}$
such that $H\cong K_{1,3}\cup_0 K_{1,2}$,
contradicting Lemma~\ref{lem5}.

Now $k=2$ and $C$ is the cycle
$a_1b_1a_2b_2a_1$.
For each $i\in [2]$,
since $d_G(b_i)\ge 3$,
$b_i$ is adjacent to some
vertex $y_i$ in $\cup_{P_j\in \setq_1}P_j^o$.
%Since $G$ is $K_{1,4}$-free,$|N_G(y)\cap I|\le 3$for each $y\in S$,implying that $y_1\ne y_2$.
Let \(L_i=N_{\setq}(y_i)\) be the pair of neighbors of \(y_i\) on its
alternating path. If \(y_1=y_2\), then this vertex has at least four
neighbors in \(I\), namely \(b_1,b_2\) and its two path neighbors.
The cycle and the cover paths are vertex-disjoint, so the four listed neighbors are distinct. Since they all lie in \(I\), they induce a \(K_{1,4}\) together with their common center \(y_1\),
a contradiction.
Thus \(y_1\ne y_2\). If \(L_1=L_2\), the four vertices would
form a cycle in the union of the cover paths. Hence we may choose
\(x\in L_2\setminus L_1\).
The stars with centers \(y_1,y_2\) and leaf sets \(L_1\cup\{b_1\}\)
and \(\{x,b_2\}\), respectively, are vertex-disjoint. Their union as
a subgraph contradicts Lemma~3.1.

%It follows that$G$ has a subgraph $H$ withthe property that$V(H)\cap S=\{y_1,y_2\}$and $H\cong K_{1,3}\cup_0 K_{1,2}$,contradicting Lemma~\ref{lem5}.

Hence Claim~\ref{cl3-2} holds.
\clproofend

We are now going to complete the proof of Lemma~\ref{lem7}.

By Claims~\ref{cl2} and~\ref{cl3-2}, $S_2=\emptyset$
and $\setq_1=\emptyset$.
Since $|V(G)|\ge 9$, we have
$2k=|V(G)|\ge 9$,
implying that $k\ge 5$.
  By Lemma \ref{lem6},
  $k\le 5$, and thus $k=5$.
  Thus, $C$ is the cycle
  $a_1b_1a_2b_2\ldots a_{5}b_{5}a_1$.

  Since $d_G(b_3)\ge 3$,
  $N_G(b_3)\cap \{a_1,a_2,a_5\}\ne\emptyset$.
  If $a_2\in N_G(b_3)$,
  then $G$ has a
  subgraph $H$ with
  the property that
  $V(H)\cap S=\{a_2,a_5\}$
  and $H\cong K_{1,3}\cup_{0} K_{1,2}$,
  contradicting Lemma~\ref{lem5}.
  Thus, $a_2\notin N_G(b_3)$.
  Similarly,
$a_5\notin N_G(b_3)$.
Then,
$N_G(b_3)\cap \{a_1,a_2,a_5\}\ne\emptyset$
implies that
$a_1\in N_G(b_3)$.

Similarly, it can be shown that
$a_2\in N_G(b_4)$.
Then, $G$ contains a subgraph
$H$ with
$
V(H)=\{a_1,a_2,b_1,b_2,b_3,b_4,b_5\}
$
and
$$
E(H)=\{a_1b_1,a_1b_3,a_1b_5,
a_2b_2,a_2b_4\}.
$$
Clearly, $|V(H)\cap S|=2$
and $H\cong K_{1,3}\cup_0 K_{1,2}$,
contradicting
Lemma~\ref{lem5}.

Hence Lemma~\ref{lem7} holds.
\proofend

\gap

The following lemma states that each path of an $I$-cover in a $3$-connected $\{K_{1,4},K_{1,4}+e\}$-free split graph has order at most 11.

\begin{lem}  \label {lem8}
Let $\mathcal P=\{P_{1},\ldots,P_{h}\}$ be an $I$-cover of a $3$-connected $\{K_{1,4},K_{1,4}+e\}$-free split graph with a split partition $(S,I)$. Then   $|V(P_{i})|\leq 11$ for each $i\in [h]$.
\end{lem}

\proof  Suppose that
$P_i$ is a path of order at least
$13$ for some $i\in [h]$.
Write its initial vertices as
\[
a_1b_1a_2b_2a_3b_3a_4b_4a_5b_5a_6b_6a_7.
\]
The three stars centered at \(a_2,a_4,a_6\), with respective leaf pairs \(\{b_1,b_2\},\{b_3,b_4\},\{b_5,b_6\}\), form a subgraph isomorphic to \(3K_{1,2}\) with exactly three vertices in \(S\), contradicting Lemma \ref{lem6}.
Hence Lemma~\ref{lem8} holds.
\proofend

\section{Proof of Theorem \ref{Th1}}

%In this section, we focus onproving Theorem \ref{Th1}.

\begin{lem}\label{le5-1}
Let $G$ be a $3$-connected $\{K_{1,4},K_{1,4}+e\}$-free split graph with a split partition $(S,I)$ and $|V(G)|\geq 13$,
and let $\setp=\{P_i: i\in [h]\}$
be an $I$-cover.
Then,
\begin{enumerate}[(i)]
	\item either $S\not\subseteq V(\setp)$ or $h\ge 2$; and
	\item
for any $i\in [h]$,
if $u$ and $v$ are the two endpoints
of $P_i$, then $G$ has a
Hamiltonian $(u,v)$-path.
\end{enumerate}
\end{lem}

\proof  (i) Suppose that
$S\subseteq V(\setp)$
and $h=1$.
Then $V(G)=S\cup I=V(\setp)
=V(P_1)$,
implying that $|V(P_1)|
=|V(G)|\ge 13$
and $G$ contains a subgraph $H$ with the property that
$|V(H)\cap S|=3$ and $H\cong 3K_{1,2}$,
contradicting Lemma~\ref{lem6}.
Thus, (i) holds.

(ii) %The proof is by contradiction. Suppose
Suppose, for a contradiction,
that $G$ does not have a Hamiltonian $(u,v)$-path.
Let $S_{1}= V(\mathcal P) \cap S$ and $S_{2}=S\setminus S_{1}$.
 For any $i\in [h]$,  let
$P_{i}=a_{1}^{i}b_{1}^{i}a_{2}^{i}b_{2}^{i}\ldots a_{k_{i}}^{i}b_{k_{i}}^{i}a_{k_{i}+1}^{i}$, where $a_{1}^{i},a_{2}^{i},\ldots, a_{k_{i+1}}^{i}\in S_{1}$, $b_{1}^{i},b_{2}^{i},\ldots, b_{k_{i}}^{i}\in I$.
Without loss of generality,
suppose that  $u$ and $v$ are
the endpoints of $P_{1}$, i.e.,
$u=a_{1}^{1}$, and $v=a_{k_1+1}^{1}$.

\setcounter{cl}{0}

\begin{cl}\label{le51-cl1}
	$N_{G}(b_{i}^{1}) \cap
	((End(\mathcal P)\setminus \{u,v\})\cup S_2)=\emptyset$
	for each $i\in [k_{1}]$.
\end{cl}

\clproof
Suppose that $b_{i}^{1}c\in E(G)$ for some
$c\in (End(\mathcal P)\setminus \{u,v\})\cup S_2$.
Then
$\mathcal P^{'}=\mathcal P -\{a_{i}^{1}b_{i}^{1}\}+\{b_{i}^{1}c\}$ is a new $I$-cover of $G$ such that $u,v\notin Inn(\mathcal P^{'})$, and
no path of $\mathcal P^{'}$
has endpoint set
$\{u,v\}$.
%are not the endpoints of any path of $\mathcal P^{'}$.
By Lemma \ref{lem9}, $G$ has a Hamiltonian $(u,v)$-path, a contradiction.  Thus, Claim  \ref{le51-cl1} holds. \clproofend

\begin{cl}\label{le51-cl2}
	$k_{1}=1$.
\end{cl}

\clproof Suppose that
$k_{1}\geq 2$. By Lemma \ref{lem8}, $|V(P_{1})|\leq 11$. Since $|V(G)|\geq 13$, $|(End(\mathcal P)\setminus \{u,v\})\cup S_2|\geq2$. By Claim \ref{le51-cl1}, there exist two distinct vertices $c_{1},c_{2}\in (End(\mathcal P)\backslash \{u,v\})\cup S_{2}$ such that $G[\{a_{2}^{1},b_{1}^{1},b_{2}^{1},c_{1},c_{2}\}]\cong K_{1,4}+e$, a contradiction. Thus, Claim  \ref{le51-cl2} holds. \clproofend

\begin{figure}[htbp]
	\centering

\begin{tikzpicture}[
	x=9.8mm, y=12.4mm,
	edge/.style={line width=.70pt, line cap=round, line join=round},
	thin edge/.style={edge, line width=.40pt},
	top label/.style={anchor=south west, inner sep=0pt, xshift=-.5mm, yshift=.9mm},
	bottom label/.style={anchor=north west, inner sep=0pt, xshift=-.9mm, yshift=-.8mm},
	path label/.style={anchor=base, inner sep=0pt}
	]
	% Vertices of P_1 and P_2, listed from left to right.
	\foreach \vertex/\x/\y in {
		u/0/0, b1/1/1, v/2/0,
		p2-start/3.34/0, p2-first-peak/4.34/1, p2-first-valley/5.34/0,
		b-prev/6.34/1, c/7.34/0, b-current/8.34/1,
		p2-last-valley/9.34/0, p2-last-peak/10.34/1, p2-end/11.34/0
	} {
		\coordinate (\vertex) at (\x,\y);
		\fill (\vertex) circle[radius=1.65pt];
	}
	
	% Edges along P_1 and P_2.
	\draw[thin edge] (u) -- (b1) (b-prev) -- (c);
	\draw[edge] (b1) -- (v);
	\draw[edge] (p2-start) -- (p2-first-peak) -- (p2-first-valley) -- (b-prev);
	\draw[edge] (c) -- (b-current) -- (p2-last-valley) -- (p2-last-peak) -- (p2-end);
	
	% Long edges between the paths.
	\draw[edge] (u) -- (b-prev) (b1) -- (c);
	
	\foreach \x in {5.04,5.34,5.64} {
		\fill (\x,.58) circle[radius=.62pt];
	}
	\foreach \x in {9.04,9.34,9.64} {
		\fill (\x,.58) circle[radius=.76pt];
	}
	
	\node[top label] at (b1) {$b_1^1$};
	\node[top label] at (b-prev) {$b_{i-1}^2$};
	\node[top label] at (b-current) {$b_i^2$};
	
	\node[bottom label] at (u) {$u=a_1^1$};
	\node[bottom label, xshift=-2.8mm] at (v) {$v=a_2^1$};
	\node[bottom label] at (c) {$c=a_i^2$};
	
	\node[path label] at (.23,.71) {$P_1$};
	\node[path label] at (11.34,.71) {$P_2$};
\end{tikzpicture}

	\caption{The case when $a_{1}^1b_{i-1}^2\in E(G)$.}
\end{figure}

Consider $b_{1}^{1}$.
Since $d_{G}(b_{1}^{1})\geq 3$,
Claims \ref{le51-cl1} and \ref{le51-cl2} imply that
$b_{1}^{1}c\in E(G)$ for some
vertex $c\in Inn(\mathcal P)$. Without loss of generality,
we assume that $c=a_{i}^{2}$. Since $G[\{a_{1}^1,b_{1}^1,a_{i}^{2},b_{i-1}^2,b_{i}^2\}]$ cannot
be isomorphic to $K_{1,4}+e$,
we have $a_{1}^1b_{i-1}^2\in E(G)$ or  $a_{1}^1b_{i}^2\in E(G)$.
Without loss of generality,
we assume that $a_{1}^1b_{i-1}^2\in E(G)$.
Then $\mathcal P^{'}=\mathcal P-\{a_{1}^{1}b_{1}^{1},a_{i}^{2}b_{i-1}^{2}\}+\{b_{1}^{1}a_{i}^{2},a_{1}^{1}b_{i-1}^{2}\}$ (depicted in Fig. 4) is a new  $I$-cover  of $G$ in which $u$ and $v$ are endpoints of different paths. By Lemma \ref{lem9}, $G$ has a Hamiltonian $(u,v)$-path, a contradiction.  Hence, (ii) holds.

%Thus, we complete the proof of
This proves the lemma.
% Lemma  \ref{le5-1}.
\proofend

%or $\mathcal P^{'}=\mathcal P-\{a_{1}^{1}b_{1}^{1},a_{i}^{2}b_{i}^{2}\}+\{b_{1}^{1}a_{i}^{2},a_{1}^{1}b_{i}^{2}\}$

\gap

\noindent{{\bf Proof of Theorem \ref{Th1}}.}
Suppose, for a contradiction,
 that $G$ is not Hamilton-connected. Then there exist
two distinct vertices
$u,v\in V(G)$ such that $G$ does not have a Hamiltonian $(u,v)$-path. Let $(S,I)$ be a split partition of $G$.
By Lemma \ref{lem7}, $G$ contains an $I$-cover
$\mathcal P=\{P_{1},\ldots,P_{h}\}$.

Since $G$ does not have a Hamiltonian $(u,v)$-path,
by Lemmas \ref{lem9} and \ref{le5-1},
every $I$-cover $\cP$
satisfies
$\{u,v\}\cap Inn(\cP)\ne \varnothing$.
%no $I$-cover $\cP'$ satisfies $\{u,v\}\cap Inn(\cP')=\varnothing$. In particular, $\{u,v\}\cap Inn(\cP)\ne\varnothing$ for every $I$-cover $\cP$.

In the following,
choose an $I$-cover $\cP=\{P_1,\ldots,P_h\}$ minimizing $|\{u,v\}\cap Inn(\cP)|$ and, subject to this, maximizing $h$.

Let $S_{1}=V(\mathcal P)\cap S$ and $S_{2}=S\setminus S_{1}$.
Let $P_{i}=a_{1}^{i}b_{1}^{i}a_{2}^{i}b_{2}^{i}\ldots a_{k_{i}}^{i}b_{k_{i}}^{i}a_{k_{i}+1}^{i}$, where $a_{1}^{i},a_{2}^{i},\ldots, a_{k_{i+1}}^{i}\in S_{1}$, $b_{1}^{i},b_{2}^{i},\ldots, b_{k_{i}}^{i}\in I$, and $i\in[h]$.

Without loss of generality, we assume that $u=a_{i_{0}}^{1}$, where
$1< i_{0} < k_{1}+1$.

\setcounter{cl}{0}
\begin{cl} \label {cl4-1}
\text{ For any} $1\le i\le h$
and
$c\in I\cap V(P_{i})$,
if  $|V(P_{i})|\geq5$,
then  $N_{G}(c)\cap S_{2}=\emptyset$.
	\end{cl}

\clproof Suppose that $b^i_jz\in E(G)$ for some $z\in S_2$ and $k_i\ge2$. If $j\ge2$, replace $a^i_jb^i_j$ by $zb^i_j$; if $j=1$, replace $b^i_1a^i_2$ by $b^i_1z$. In either case, $P_i$ splits into two nontrivial alternating paths, and all vertices of $I$ remain covered. No vertex becomes interior: $z$ is a new endpoint, and one formerly interior clique vertex becomes an endpoint. If the number of interior vertices from $\{u,v\}$ decreases, the first choice criterion is contradicted. If it is unchanged, the number of paths increases by one, contradicting the second choice criterion. Thus $N_G(b^i_j)\cap S_2=\varnothing$.

%Since $h$ is maximized, Claim \ref{cl4-1} follows immediately.
\clproofend

\begin{cl} \label {cl4-2}
If $|N_{G}(d)\cap I|=3$
for some $d\in S_{1}$,
then $S_{2}=\emptyset$.
\end{cl}
\clproof
Suppose that $|N_{G}(d)\cap I|=3$ for some $d\in S_{1}$
and $S_{2}\ne \emptyset$.
Let $z\in S_{2}$.

Since $u\in Inn(P_1)$,
$|V(P_{1})|\geq5$.
By Lemma \ref{lem5},
$1\leq |N_{G}(d)\cap I \cap V(P_{1})|\leq3$.
If $|N_{G}(d)\cap I \cap V(P_{1})|\ge 2$, by Claim \ref{cl4-1}, $G[\{z,d\}\cup (N_{G}(d)\cap I)]\cong K_{1,4}$ or $K_{1,4}+e$,
a contradiction.
Thus, $|N_{G}(d)\cap I \cap V(P_{1})|=1$.

We assume that
$b_{j}^{1}\in N_{G}(d)\cap I \cap V(P_{1})$ and $\{x,y\}=(N_{G}(d)\cap I) \setminus \{b_{j}^{1}\}$. We choose $a_{j_{0}}^{1}=a_{2}^{1}$ if $j=1$ and  $a_{j_{0}}^{1}=a_{j}^{1}$ if $1<j\leq k_{1}$. Note that $|N_{G}(a_{j_{0}}^{1})\cap I\cap V(P_{1}
)|\geq 2$. Then $x,y\notin N_{G}(a_{j_{0}}^{1})$; otherwise,  $|N_{G}(a_{j_{0}}^{1})\cap I |\geq 3$,  contradicting  the fact that $G$ is $K_{1,4}$-free or $|N_{G}(d)\cap I \cap V(P_{1})|=1$ for any $d\in S_{1}$ with $|N_{G}(d)\cap I|=3$. However, $G[\{d,b_{j}^{1},a_{j_{0}}^{1},x,y\}]\cong K_{1,4}+e$, a contradiction. Thus, Claim \ref{cl4-2} holds.
\clproofend

\begin{cl} \label {cl16}
We have
\begin{align*}
	N_G(b^1_{i_0})\cap
	\bigl(End(\cP)\setminus\{v,a^1_{k_1+1}\}\bigr)
	=\varnothing,
	\qquad
	N_G(b^1_{i_0-1})\cap
	\bigl(End(\cP)\setminus\{v,a^1_1\}\bigr)=\varnothing.
\end{align*}
\end{cl}

\clproof
Suppose first that $z\in End(\cP)\setminus\{v,a^1_{k_1+1}\}$ and $zb^1_{i_0}\in E(G)$. Thus $z$ is either $a^1_1$ or an endpoint of another path. Replace the edge $a^1_{i_0}b^1_{i_0}$ by $zb^1_{i_0}$. If $z=a^1_1$, the two pieces of $P_1$ are rejoined into a path whose endpoint is $u=a^1_{i_0}$. If $z$ belongs to another path, that path joins the suffix of $P_1$, while the prefix of $P_1$ ends at $u$. In either case, the resulting collection is an $I$-cover.

The vertex $u$ is no longer interior. The only endpoint that becomes interior is $z$, and $z\ne v$. Therefore
\[
|\{u,v\}\cap Inn(\cP')|<|\{u,v\}\cap Inn(\cP)|,
\]
contrary to the choice of $\cP$. The second assertion follows by reversing $P_1$.
\clproofend

\begin{cl} \label {cl5}

 If $v\neq a_{1}^{1}$, then $N_{G}(b_{i_{0}}^{1})\subseteq \{v,a_{i_{0}}^{1}, a_{i_{0}+1}^{1},\ldots,a_{k_{1}+1}^{1}\}$;
 if $v\neq a_{k_{1}+1}^{1}$, then $N_{G}(b_{i_{0}-1}^{1})\subseteq \{v,a_{1}^{1}, a_{2}^{1},\ldots,a_{i_{0}}^{1}\}$.
\end{cl}

\begin{figure}[htbp]
  \centering
   %Requires \usepackage{graphicx}
   
   \begin{tikzpicture}[
   	x=6.75mm, y=8.9mm, font=\small,
   	edge/.style={line width=.63pt, line cap=round, line join=round},
   	thin edge/.style={edge, line width=.36pt},
   	top label/.style={anchor=south west, inner sep=0pt, xshift=-.6mm, yshift=.8mm},
   	bottom label/.style={anchor=north west, inner sep=0pt, xshift=-.7mm, yshift=-.7mm},
   	path label/.style={anchor=base, inner sep=0pt}
   	]
   	% Vertices of P_1 and P_2, listed from left to right.
   	\foreach \vertex/\x/\y in {
   		a-first/0/0, p1-first-peak/1/1, p1-first-valley/2/0,
   		b-i-prev/3/1, u/4/0, b-i/5/1,
   		p1-last-valley/6/0, p1-last-peak/7/1, p1-end/8/0,
   		p2-start/8.64/0, p2-first-peak/9.64/1, p2-first-valley/10.64/0,
   		b-j-prev/11.64/1, z/12.64/0, b-j/13.64/1,
   		p2-last-valley/14.64/0, p2-last-peak/15.64/1, p2-end/16.64/0
   	} {
   		\coordinate (\vertex) at (\x,\y);
   		\fill (\vertex) circle[radius=1.32pt];
   	}
   	
   	% Edges along P_1 and P_2.
   	\draw[edge] (a-first) -- (p1-first-peak) -- (p1-first-valley) -- (b-i-prev) -- (u);
   	\draw[edge] (b-i) -- (p1-last-valley) -- (p1-last-peak) -- (p1-end);
   	\draw[edge] (p2-start) -- (p2-first-peak) -- (p2-first-valley) -- (b-j-prev) -- (z);
   	\draw[edge] (b-j) -- (p2-last-valley) -- (p2-last-peak) -- (p2-end);
   	\draw[thin edge] (u) -- (b-i) (z) -- (b-j);
   	
   	% Long edges between the paths.
   	\draw[edge] (a-first) -- (b-j) (b-i) -- (z);
   	
   	\foreach \middle in {2,6,10.64,14.64} {
   		\foreach \offset in {-.30,0,.30} {
   			\fill ({\middle+\offset},.57) circle[radius=.62pt];
   		}
   	}
   	
   	\node[top label] at (b-i-prev) {$b_{i_0-1}^1$};
   	\node[top label] at (b-i) {$b_{i_0}^1$};
   	\node[top label] at (b-j-prev) {$b_{j-1}^2$};
   	\node[top label] at (b-j) {$b_j^2$};
   	
   	\node[bottom label] at (a-first) {$a_1^1$};
   	\node[bottom label] at (u) {$u=a_{i_0}^1$};
   	\node[bottom label] at (z) {$z=a_j^2$};
   	
   	\node[path label] at (8.02,.89) {$P_1$};
   	\node[path label] at (16.64,.89) {$P_2$};
   \end{tikzpicture}

  \caption{The case when $z=a_{j}^{2}$ in the proof of Claim \ref{cl5}.}
\end{figure}

\clproof  By symmetry, it is sufficient to prove  $N_{G}(b_{i_{0}}^{1})\subseteq \{v,a_{i_{0}}^{1}, a_{i_{0}+1}^{1},\ldots,a_{k_{1}+1}^{1}\}$ if $v\neq a_{1}^{1}$.

Suppose that $v\neq a_{1}^{1}$
and
$z\in N_{G}(b_{i_{0}}^{1})\setminus \{v,a_{i_{0}}^{1}, a_{i_{0}+1}^{1},\ldots,a_{k_{1}+1}^{1}\}$. By the choice of $\mathcal P$, $a_{1}^{1}b_{i_{0}}^{1}\notin E(G)$. By Claim \ref{cl4-1}, $z\notin S_{2}$. Then, by Claim \ref{cl16},   $z\notin End (\mathcal P)\backslash \{v\}$. Thus, $z\in Inn(\mathcal P)\setminus \{a_{i_{0}}^{1},a_{i_{0}+1}^{1}\}$. By Claim \ref{cl4-2}, $S_{2}=\emptyset$. Then, by Lemma \ref{le5-1}(i), $h\geq2$. Without loss of generality, we assume that $v\neq a_{1}^{2}$. Then $b_{i_{0}}^{1}a_{1}^{2}\notin E(G)$.

We further suppose that $z=a_{j}^{2}\in V(P_{2})\cap Inn(\mathcal P)$. Consider $G[\{a_{j}^{2},a_{1}^{1},b_{i_{0}}^{1}, b_{j-1}^{2},b_{j}^{2}\}]$. Since $G$ is $\{K_{1,4},K_{1,4}+e\}$-free, $a_{1}^{1}b_{j-1}^{2},a_{1}^{1}b_{j}^{2}\in E(G)$. Then $\mathcal P'=\mathcal P-\{a_{i_{0}}^{1}b_{i_{0}}^{1},a_{j}^{2}b_{j}^{2}\}+\{a_{1}^{1}b_{j}^{2},a_{j}^{2}b_{i_{0}}^{1}\}$ is a new $I$-cover with
$|\{u,v\}\cap Inn(\mathcal P')|< |\{u,v\}\cap Inn(\mathcal P)|$ (depicted in Fig. 5), contrary to the choice of $\mathcal P$. Then $N_{G}(b_{i_{0}}^{1})\subseteq V(P_{1})\cup \{v\}$ and $z=a_{j_{0}}^{1}$ for some $j_{0}\in \{2,3,\ldots,i_{0}-1\}$.

\begin{figure}[htbp]
  \centering

   \begin{tikzpicture}[
   	x=6.6mm, y=8.65mm, font=\small,
   	edge/.style={line width=.63pt, line cap=round, line join=round},
   	thin edge/.style={edge, line width=.36pt},
   	top label/.style={anchor=south west, inner sep=0pt, xshift=-.6mm, yshift=.8mm},
   	bottom label/.style={anchor=north west, inner sep=0pt, xshift=-.7mm, yshift=-.7mm},
   	path label/.style={anchor=base, inner sep=0pt}
   	]
   	% Vertices of P_2 and P_1, listed from left to right.
   	\foreach \vertex/\x/\y in {
   		a-first/0/0, p2-first-peak/1/1, p2-valley/2/0, p2-last-peak/3/1, p2-end/4/0,
   		p1-start/5/0, p1-first-peak/6/1, p1-first-valley/7/0, b-j-prev/8/1,
   		z/9/0, b-j/10/1, p1-middle-valley/11/0, b-i-prev/12/1,
   		u/13/0, b-i/14/1, p1-last-valley/15/0, p1-last-peak/16/1, p1-end/17/0
   	} {
   		\coordinate (\vertex) at (\x,\y);
   		\fill (\vertex) circle[radius=1.32pt];
   	}
   	
   	% Edges along P_2 and P_1.
   	\draw[edge] (a-first) -- (p2-first-peak) -- (p2-valley) -- (p2-last-peak) -- (p2-end);
   	\draw[edge] (p1-start) -- (p1-first-peak) -- (p1-first-valley) -- (b-j-prev);
   	\draw[edge] (z) -- (b-j) -- (p1-middle-valley) -- (b-i-prev) -- (u);
   	\draw[edge] (b-i) -- (p1-last-valley) -- (p1-last-peak) -- (p1-end);
   	\draw[thin edge] (b-j-prev) -- (z) (u) -- (b-i);
   	
   	% Long edges incident with the labelled vertices.
   	\draw[edge] (a-first) -- (b-j-prev) (z) -- (b-i);
   	
   	\foreach \middle in {2,7,11,15} {
   		\foreach \offset in {-.30,0,.30} {
   			\fill ({\middle+\offset},.57) circle[radius=.62pt];
   		}
   	}
   	
   	\node[top label] at (b-j-prev) {$b_{j_0-1}^1$};
   	\node[top label] at (b-j) {$b_{j_0}^1$};
   	\node[top label] at (b-i-prev) {$b_{i_0-1}^1$};
   	\node[top label] at (b-i) {$b_{i_0}^1$};
   	
   	\node[bottom label] at (a-first) {$a_1^2$};
   	\node[bottom label] at (z) {$z=a_{j_0}^1$};
   	\node[bottom label] at (u) {$u=a_{i_0}^1$};
   	
   	\node[path label] at (3.91,.87) {$P_2$};
   	\node[path label] at (17,.89) {$P_1$};
   \end{tikzpicture}

  \caption{The case when $z=a_{j_{0}}^{1}$ in the proof of Claim \ref{cl5}.}
\end{figure}

Consider $G[\{a_{j_{0}}^{1},b_{j_{0}-1}^{1},b_{j_{0}}^{1},b_{i_{0}}^{1},a_{1}^{2}\}]$. Since $G$ is $\{K_{1,4},K_{1,4}+e\}$-free, $a_{1}^{2}b_{j_{0}-1}^{1},a_{1}^{2}b_{j_{0}}^{1}\in E(G)$. Then $\mathcal P'=\mathcal P-\{a_{j_{0}}^{1}b_{j_{0}-1}^{1},a_{i_{0}}^{1}b_{i_{0}}^{1}\}+\{a_{1}^{2}b_{j_{0}-1}^{1},a_{j_{0}}^{1}b_{i_{0}}^{1}\}$ is a new $I$-cover with
$|\{u,v\}\cap Inn(\mathcal P')|< |\{u,v\}\cap Inn(\mathcal P)|$ (depicted in Fig. 6), contrary to the choice of $\mathcal P$. Thus, Claim \ref{cl5}  holds. \clproofend

\begin{cl} \label {cl7}

$v\in S_{1}$.
\end{cl}
\clproof  %By a way of contradiction.
Suppose that $v\notin S_{1}$. Then $v\in I\cup S_{2}$.
By Claims \ref{cl4-1}, \ref{cl4-2}
and \ref{cl5},
$N_{G}(b_{i_{0}}^{1})\subseteq \{a_{i_{0}}^{1}, a_{i_{0}+1}^{1},\ldots,a_{k_{1}+1}^{1}\}$ and
$N_{G}(b_{i_{0}-1}^{1})\subseteq \{a_{1}^{1}, a_{2}^{1},\ldots,a_{i_{0}}^{1}\}$. Then $h=1$ and $|S_{2}|\leq 1$. Otherwise,  there exist two distinct vertices $z_{1},z_{2}\in (S_{1}\cap V(P_{2}))\cup S_{2}$ such that the subgraph of $G$
induced by  $\{a_{i_{0}}^{1},b_{i_{0}-1}^{1},b_{i_{0}}^{1},z_{1},z_{2}\}
$ is isomorphic to
$K_{1,4}+e$, a contradiction. Since \(h=1\) and \(|S_2|\le1\), we have \(|V(P_1)|=|V(G)|-|S_2|\ge13-1=12\). Since \(|V(P_1)|\) is odd, it follows that \(|V(P_1)|\ge13\), contradicting Lemma \ref{lem8}. Thus, Claim \ref{cl7}  holds. \clproofend

\begin{cl} \label {cl9}

$v\in Inn(\mathcal P)$.
\end{cl}
\clproof Suppose that $v\notin Inn(\mathcal P)$.
Then, $v\in End(\mathcal P)$ by Claim \ref{cl7}.
Without loss of generality, we assume that $v\neq a_{1}^{1}$. By Claim \ref{cl5}, $N_{G}(b_{i_{0}}^{1})\subseteq  \{v,a_{i_{0}}^{1}, a_{i_{0}+1}^{1},\ldots,a_{k_{1}+1}^{1}\}$.

Assume that $v\neq a_{k_{1}+1}^{1}$. Then $h\geq 2$. By Lemma \ref{lem8}, $|V(P_{1})|\leq 11$. Then $h=2$ and $S_{2}=\emptyset$. Otherwise,  by Claims \ref{cl4-1}, \ref{cl16} and \ref{cl5}, there exist two distinct vertices $z_{1},z_{2}\in (End(\mathcal P)\backslash (V(P_{1})\cup \{v\}))\cup S_{2}$ such that $G[\{a_{i_{0}}^{1},b_{i_{0}-1}^{1},b_{i_{0}}^{1},z_{1},z_{2}\}]\cong K_{1,4}+e$, a contradiction. Then $|V(P_{2})|=3$. Otherwise,  by Claim \ref{cl5}, there exist two distinct vertices $z_{1},z_{2}\in (V(P_{2})\cap S_{1})\backslash \{v\}$ such that $G[\{a_{i_{0}}^{1},b_{i_{0}-1}^{1},b_{i_{0}}^{1},z_{1},z_{2}\}]\cong K_{1,4}+e$, a contradiction.

Since \(h=2\), \(S_2=\varnothing\) and \(|V(P_2)|=3\), we have \(|V(P_1)|=|V(G)|-3\ge10\). Since \(|V(P_1)|\) is odd, Lemma~\ref{lem8} gives \(|V(P_1)|=11\).
%We have \(|V(P_1)|=|V(G)|-3\ge10\). Since \(|V(P_1)|\) is odd, Lemma 4.2 gives \(|V(P_1)|=11\).
%So, $|V(P_{1})|=11$ since $|V(G)|\geq 13$ and $S_{2}=\emptyset$.
Temporarily write
\[
P_1=a_1b_1a_2b_2a_3b_3a_4b_4a_5b_5a_6,
\qquad x=b^2_1.
\]
If \(xa_j\in E(G)\), the corresponding row
in the table below
lists three vertex-disjoint stars that form a subgraph isomorphic to \(3K_{1,2}\). In the table, \(c\colon L\) denotes the star with center \(c\) and leaf set \(L\).
\begin{center}
	\renewcommand{\arraystretch}{1.12}
	\begin{tabular}{@{}cccc@{}}
		\toprule
		$j$ & Star 1 & Star 2 & Star 3 \\
		\midrule
		$1,2$ & $a_j\colon\{x,b_1\}$ & $a_3\colon\{b_2,b_3\}$
		& $a_5\colon\{b_4,b_5\}$ \\
		$3,4$ & $a_j\colon\{x,b_3\}$ & $a_2\colon\{b_1,b_2\}$
		& $a_5\colon\{b_4,b_5\}$ \\
		$5,6$ & $a_j\colon\{x,b_5\}$ & $a_2\colon\{b_1,b_2\}$
		& $a_4\colon\{b_3,b_4\}$ \\
		\bottomrule
	\end{tabular}
\end{center}
By Lemma \ref{lem6}, \(b_1^2\) has no neighbor on \(P_1\). Since \(h=2\) and \(S_2=\varnothing\), its only neighbors are the two clique vertices of \(P_2\). Hence \(d_G(b_1^2)=2\), contradicting \(\delta(G)\ge3\). Therefore \(v=a^1_{k_1+1}\).

Now $h=1$ and $|S_{2}|\leq 1$.  Otherwise, by Claims \ref{cl4-1} and \ref{cl16}, there exist two distinct vertices $z_{1},z_{2}\in (End(\mathcal P)\backslash V(P_{1}))\cup S_{2}$ such that $G[\{a_{i_{0}}^{1},b_{i_{0}-1}^{1},b_{i_{0}}^{1},z_{1},z_{2}\}]\cong K_{1,4}+e$, a contradiction.
Since \(h=1\) and \(|S_2|\le1\), we have \(|V(P_1)|=|V(G)|-|S_2|\ge13-1=12\). Since \(|V(P_1)|\) is odd, it follows that \(|V(P_1)|\ge13\), contradicting Lemma \ref{lem8}.
Thus, Claim \ref{cl9}  holds.\clproofend

By Claim \ref{cl9}, $v\neq a_{1}^{1}$ and $v\neq a_{k_{1}+1}^{1}$. Then, by Claim \ref{cl5}, we have $N_{G}(b_{i_{0}-1}^{1})\subseteq \{v,a_{1}^{1}, a_{2}^{1},\ldots,a_{i_{0}}^{1}\}$ and $N_{G}(b_{i_{0}}^{1})\subseteq \{v,a_{i_{0}}^{1}, a_{i_{0}+1}^{1},\ldots,a_{k_{1}+1}^{1}\}$.

Then $h=1$ and $|S_{2}|\leq 1$.  Otherwise, by Claims \ref{cl4-1}, \ref{cl16} and \ref{cl9}, there exist two distinct vertices $z_{1},z_{2}\in (End(\mathcal P)\backslash (V(P_{1})\cup \{v\}))\cup S_{2}$ such that $G[\{a_{i_{0}}^{1},b_{i_{0}-1}^{1},b_{i_{0}}^{1},z_{1},z_{2}\}]\cong K_{1,4}+e$, a contradiction.
So \(|V(P_1)|\ge13-1=12\)
and  its odd order gives \(|V(P_1)|\ge13\),
%So, $|V(P_{1})|\geq 13$ since $|V(G)|\geq 13$ and $|S_{2}|\leq 1$,
contradicting Lemma \ref{lem8}.

This proves the theorem.
\proofend

\section{Further study}

%\Red{The whole section is revised.}

\def \setc {{\mathcal C}}

For each integer $r\ge 3$,
let $\setc_r$
be the set of $r$-connected $\{K_{1,r+1},K_{1,r+1}+e\}$-free split graphs.
Theorem~\ref{Th1}
shows that
every graph in $\setc_3$
of order at least $13$ is
Hamilton-connected.
It is natural to ask whether,
for each fixed integer \(r\ge4\), every sufficiently large graph in \(\mathcal C_r\) is Hamilton-connected.

For each integer \(r\ge3\), let \(n_{\rmhc}(r)\) be the least positive integer \(n\) such that every graph \(G\in\mathcal C_r\) with \(|V(G)|\ge n\) is Hamilton-connected.
If no such \(n\) exists, set \(n_{\rmhc}(r)=\infty\). By Theorem~\ref{Th1}, \(n_{\rmhc}(3)\le13\).

We now pose the following problem.

\begin{prob} \label {con4}
	Determine
	$n_{\rmhc}(r)$ for $r\ge 3$.
\end{prob}

%It is natural to consider whether Theorem~\ref{Th1}can be extended to any graph $G$ contained in $\setc_r$  for $r\ge 4$.

In the following, we
construct a $22$-vertex graph in $\setc_5$ that
fails to be Hamilton-connected.
Thus, $n_{\rmhc}(5)\ge 23$.

\begin{figure}[h!]
	\centering
\begin{tikzpicture}[scale=0.7,
	line cap=round,
	clique edge/.style={draw=black!32,line width=0.32pt},
	incidence edge/.style={draw=blue!58!black,line width=0.46pt},
	clique vertex/.style={circle,draw=black,fill=black,
		inner sep=0pt,minimum size=4.2pt},
	independent vertex/.style={circle,draw=blue!58!black,fill=white,
		line width=0.8pt,inner sep=0pt,minimum size=5.2pt},
	vertex label/.style={font=\normalsize,fill=white,inner sep=1.2pt}
	]
	% Coordinates: S on the inner circle; I on the outer circle.
	% The small angular offset avoids radial alignments.
	\foreach \i in {0,...,10}{
		\coordinate (s\i) at ({90-360*\i/11}:3.05);
		\coordinate (x\i) at ({98-360*\i/11}:5.55);
	}
	
	% All 55 edges of G[S] = K_11.
	\foreach \i in {0,...,9}{
		\pgfmathtruncatemacro{\nexti}{\i+1}
		\foreach \j in {\nexti,...,10}{
			\draw[clique edge] (s\i)--(s\j);
		}
	}
	
	% The 55 cross edges: s_i x_j iff j-i is in {1,3,4,5,9} mod 11.
	\foreach \i in {0,...,10}{
		\foreach \d in {1,3,4,5,9}{
			\pgfmathtruncatemacro{\j}{mod(\i+\d,11)}
			\draw[incidence edge] (s\i)--(x\j);
		}
	}
	
	% Draw vertices last; the white halos keep crossings distinct from vertices.
	\foreach \i in {0,...,10}{
		\fill[white] (s\i) circle[radius=3.4pt];
		\node[clique vertex] at (s\i) {};
		\fill[white] (x\i) circle[radius=3.9pt];
		\node[independent vertex] at (x\i) {};
	}
	
	% Labels have equal font size and white backgrounds.
	\foreach \i in {0,...,10}{
		\node[vertex label] at ({90-360*\i/11}:2.55) {$s_{\i}$};
		\node[vertex label] at ({98-360*\i/11}:6.11) {$x_{\i}$};
	}
\end{tikzpicture}

\caption{A $5$-connected
$\{K_{1,6},K_{1,6}+e\}$-free split graph that is not Hamilton-connected.
}
\end{figure}

Suppose that
$G$ is a graph with
a split partition $(S,I)$,
 where
$S=\{s_i:0\le i\le 10\}$
is a clique of $G$ and
$I=\{x_j:0\le j\le 10\}$
is an independent set of $G$.
Let $\mathbb Z_{11}$ be the integers modulo $11$ and put
$
D=\{1,3,4,5,9\}\subseteq\mathbb Z_{11}$.
The edge set of $G$ is
$$
E(G)=E(G[S])\cup
\{s_ix_j:  j-i\in D\},
$$
where all subscripts and differences in $\{s_ix_j:  j-i\in D\}$
are interpreted modulo \(11\).
%%fig

%%end fig

It can be verified that $G\in \setc_5$.
%$5$-connected $\{K_{1,6},K_{1,6}+e\}$-free split graph.
Obviously,
a Hamiltonian path in $G$
with endpoints $s_i$ and $s_j$,
where $0\le i<j\le 10$,  would require at least $|I|+1 = 12$ vertices
of $S$ to separate the eleven independent vertices in $I$.
There are only eleven vertices in $S$. Hence no
such path exists, and $G$ is not Hamilton-connected.

\gap

\noindent{\bf Declaration of competing interest}

There is no conflict of interest.\\

\noindent{\bf Data availability}

No datasets were
generated or analysed
during this study.

%No data was used for the research described in the paper.

\section*{Acknowledgements}

The first author thanks the National
Institute of Education, Nanyang Technological University, Singapore, for its hospitality during this research.
%The first author would like to thank the hospitality of the National Institute of Education, Nanyang Technological University in Singapore, where the work was done.
This work was supported by  the National Natural Science Foundation of China
(Grant Nos. 12101126, 12371340) and  the Natural Science Foundation of Fujian Province (No. 2023J01539). This work was also partly supported by
the China Scholarship Council (No. 202409100010).

%\vspace{1 cm}\red{References are revised.} %%remove it

%\end{CJK}

\end{document}